\documentclass[11pt]{amsart}
\usepackage{amssymb,amsmath, amsfonts, enumerate, url} 
\usepackage{diagrams}

\newtheorem{thm}{Theorem}[section]
\newtheorem{prop}[thm]{Proposition}
\newtheorem{lem}[thm]{Lemma}
\newtheorem{cor}[thm]{Corollary}

\theoremstyle{definition}

\newtheorem{remark}[thm]{Remark}

\begin{document}

%

\newcommand\Aut{\operatorname{Aut}}
\newcommand\trdeg{\operatorname{trdeg}}
\newcommand\diag{\operatorname{diag}}
\newcommand\Diag{\operatorname{Diag}}
\newcommand\ed{\operatorname{ed}}
\newcommand\im{\operatorname{im}}
\newcommand\id{\operatorname{id}}
\newcommand\wt{\operatorname{wt}}
\newcommand\ind{\operatorname{ind}}
\newcommand\rk{\operatorname{rk}}
\newcommand\sign{\operatorname{sign}}
\newcommand\Spec{\operatorname{Spec}}
\newcommand\Spin{\operatorname{Spin}}
\newcommand\SO{\operatorname{SO}}
\newcommand\Mat{\operatorname{M}}
\newcommand\Ima{\operatorname{Im}}
\newcommand\Char{\operatorname{char}}
\newcommand\Gal{\operatorname{Gal}}
\newcommand\rank{\operatorname{rank}}
\newcommand\Ker{\operatorname{Ker}}
\newcommand\Cok{\operatorname{Coker}}

\newcommand\GL{\operatorname{GL}}
\newcommand\SL{\operatorname{SL}}
\newcommand\PGL{\operatorname{PGL}}
\newcommand\PGLn{\operatorname{PGL}_n}
\newcommand\Sym{\operatorname{S}}
\newcommand\Alt{\operatorname{A}}
\newcommand\Symn{\operatorname{S}_n}
\newcommand\Altn{\operatorname{A}_n}
\newcommand\Stab{\operatorname{Stab}}
\newcommand\End{\operatorname{End}}
\newcommand\Span{\operatorname{Span}}
\newcommand\bbC{\mathbb{C}}
\newcommand\ZZ{\mathbb{Z}}
\newcommand\bbZ{\mathbb{Z}}
\newcommand\bbG{\mathbb{G}}
\newcommand\bbQ{\mathbb{Q}}

\title[Essential dimension]{The essential dimension of the normalizer
of a maximal torus in the projective linear group}
\author[Aurel Meyer]{Aurel Meyer$^\dagger$}
\thanks{$^\dagger$ Aurel Meyer was partially supported by
a University Graduate Fellowship at the University of British Columbia}
\author[Zinovy Reichstein]{Zinovy Reichstein$^{\dagger \dagger}$}
\thanks{$^{\dagger \dagger}$ Z. Reichstein was partially supported by
NSERC Discovery and Accelerator Supplement grants}

\address{Department of Mathematics, University of British Columbia,
Vancouver, BC V6T 1Z2, Canada}



\begin{abstract} 
Let $p$ be a prime, $k$ be a field of characteristic $\neq p$ 
containing a primitive $p$th root of unity and $N$ be the normalizer 
of the maximal torus in the projective linear group $\PGLn$.
We compute the exact value of the essential dimension
$\ed_k(N; p)$ of $N$ at $p$ for every $n \ge 1$.
\end{abstract}

\maketitle
\tableofcontents

\section{Introduction}

Let $p$ be a prime, $k$ be a field of characteristic $\neq p$ 
and $N$ be the normalizer of a split maximal torus 
in the projective linear group $\PGL_n$, for some integer $n$. 
The purpose of this paper is to compute the essential 
dimension $\ed_k(N; p)$ of $N$ at $p$. For the definition 
of essential dimension of an algebraic group (and more generally, 
of a functor), we refer the reader 
to~\cite{reichstein2}, \cite{bf}, \cite{brv} or~\cite{merkurjev}.
As usual, if the reference to $k$ is clear from the context, we
will sometimes write $\ed$ in place of $\ed_k$.

We begin by explaining why we are interested in the essential 
dimension of $N$. One of the central problems in the theory 
of essential dimension is to find the exact value of 
the essential dimension of the projective linear group $\PGL_n$ 
or equivalently, of the functor
\begin{eqnarray*} 
\text{$H^1( \, \ast \, , \PGL_n) \colon K \mapsto \{$ degree $n$
central simple algebras $A/K$,} \\
\text{up to $K$-isomorphism $\}$,} 
\end{eqnarray*} 
where $K$ is a field extension of $k$.
This problem arises naturally in the theory of central simple 
algebras. To the best of our knowledge, it was first raised 
by C.~Procesi, who showed (using different terminology)
that $\ed(\PGL_n) \leq n^2$; see~\cite[Theorem 2.1]{procesi}.  
This problem, and the related question of computing the
relative essential dimension $\ed(\PGL_n; p)$ at a prime $p$, 
remain largely open. The best currently known lower bound, 
\[ \ed(\PGL_{p^r}; p) \ge 2r \]
(cf.~\cite[Theorem 16.1(b)]{reichstein1} or \cite[Theorem 8.6]{ry}), 
falls far below the best known upper bound, 
\begin{equation} \label{e.upper-bound}
\ed(\PGL_n) \le \begin{cases} 
\text{$\frac{(n-1)(n-2)}{2}$, for every odd $n \ge 5$ and} \\ 
\text{$n^2 - 3n + 1$, for every $n \ge 4$;}  
\end{cases} 
\end{equation}
see~\cite{lr}, \cite[Theorem 1.1]{lrrs},~\cite[Proposition 1.6]{lemire} 
and~\cite{ff}. 

We remark that the primary decomposition theorem reduces 
the computation of $\ed(\PGL_n; p)$ to the case where $n$ 
is a power of $p$.  That is, if $n = p_1^{r_1} \ldots p_s^{r_s}$
then $\ed(\PGL_n; p_i) = \ed(\PGL_{p_i^{r_i}}; p_i)$.
The computation of $\ed(\PGL_n)$ also
partially reduces to the prime power case, because 
\[ \ed(\PGL_{p^i}) \le \ed(\PGL_n) \le \ed(\PGL_{p_1^{r_1}}) + \ldots + 
\ed(\PGL_{p_s^{r_s}}) \]
for every $i = 1, \dots, s$; cf.~\cite[Proposition 9.8]{reichstein2}. 

It is important to note that the proofs of the upper 
bounds~\eqref{e.upper-bound} are not based on a direct 
analysis of the functor $H^1(\, \ast \, , \PGL_n)$.
Instead, one works with the related functor
\[ \text{$H^1( \, \ast \, , N) \colon K \mapsto \{$ $K$-isomorphism classes 
of pairs $(A, L)$ $\}$,} \]
where $K$ is a field extension of $k$, $A$ is a degree $n$ 
central simple algebra over $K$, $L$ is a maximal 
\'etale subalgebra of $A$, and $N$ is the normalizer of a (split)
maximal torus in $\PGL_n$. This functor is often more accessible 
than $H^1( \, \ast \, , \PGL_n)$ because many 
of the standard constructions in the theory of central simple algebras 
depend on the choice of a maximal subfield $L$ in a given central 
simple algebra $A/K$. Projecting a pair $(A, L)$ to 
the first component, we obtain a surjective morphism of functors
$H^1( \, \ast \,, N) \to H^1(\, \ast \, , \PGL_n)$. The surjectivity
of this morphism (which is a special case of a more general result 
of T. Springer (see~\cite[III.4.3, Lemma 6]{serre-gc}) 
leads to the inequalities 
\begin{equation}
\label{e.N}
\text{$\ed(N) \ge \ed(\PGL_n)$ and $\ed(N; p) \ge \ed(\PGL_n; p)$} \, ;  
\end{equation}
see~\cite[Proposition 1.3]{merkurjev}, \cite[Lemma 1.9]{bf} or
\cite[Proposition 4.3]{reichstein2}.
The inequalities~\eqref{e.upper-bound} were, in fact, proved 
as upper bounds on $\ed(N)$; see~\cite{lrrs} and~\cite{lemire}.  
It is thus natural to try to determine the exact values of 
$\ed(N)$ and $\ed(N; p)$.
In addition to being of independent interest, these 
numbers represent a limitation on the techniques used 
in~\cite{lrrs} and~\cite{lemire}. This brings us to 
the main result of this paper.

\begin{thm} \label{thm1}
Let $N$ the normalizer of a maximal torus in the projective linear 
group $\PGL_n$ defined over a field $k$. 
Assume $\Char(k) \ne p$ and $k$ contains a primitive
$p$th root of unity.  Then

\smallskip
\begin{tabular}{lll}
(a) & $\ed_k(N; p) = [n/p]$, &\quad if $n$ is not divisible by $p$. \\
(b) & $\ed_k (N; p) = 2$, &\quad if $n=p$. \\
(c) & $\ed_k (N; p) = n^2/p - n + 1$, &\quad if $n=p^r$ for some $r\ge 2$. \\
(d) & $\ed_k (N; p) = p^e(n-p^e)-n+1$, &\quad in all other cases.
\end{tabular}

\smallskip
\noindent
Here $[n/p]$ denotes the integer part of $n/p$ and
$p^e$ denotes the highest power of $p$ dividing $n$.
\end{thm}

In each part we will prove an upper bound and a lower bound on
$\ed(N)$ separately, using rather different techniques.
There is nothing about the methods we use that 
in any way guarantees that the lower bounds should
match the upper bounds, thus yielding an exact value of 
$\ed(N; p)$. The fact that this happens, under the rather mild 
requirements on $k$ imposed in the statement of 
Theorem~\ref{thm1}, may be viewed as a lucky coincidence. 

If the assumptions on $k$ are further relaxed, the upper and lower
bounds no longer match; however, some of our arguments still go through. 
In particular, all lower bounds remain valid
(i.e., $\ed_k(N)$ remains greater than or equal to 
the values specified in the theorem) over any field
$k$ of characteristic $\ne p$. 
The upper bounds on $\ed_k(N; p)$ in parts (c) and (d) are valid 
over an arbitrary field $k$ of any characteristic.

A quick glance at the statement of Theorem~\ref{thm1} shows that,
unlike in the case of $\PGLn$, the computation of $\ed(N; p)$
does not reduce to the case where $n$ is a power of $p$. On 
the other hand, the proof of part (c), where $n = p^r$ and $r \ge 2$, 
requires the most intricate arguments.
Another reason for our special interest in part (c)
is that it leads to a new upper bound on
$\ed(\PGL_n; p)$. More precisely, combining the upper bound 
in part (c) with~\eqref{e.N}, and remembering that
the upper bound in part (c) does not require any assumption 
on the ground field $k$, we obtain the following inequality.

\begin{cor} \label{cor1} Let $n = p^r$ be a prime power.  Then
\[ \ed_k(\PGL_n; p) \le p^{2r - 1} - p^r + 1 \]  
for any field $k$ and for any $r \ge 2$.
\qed
\end{cor}

Corollary~\ref{cor1} fails for $r = 1$ because 
\begin{equation} \label{e.pgl_p}
\ed_k(\PGL_p; p) \ge 2, 
\end{equation}
see~\cite[Corollary 5.7]{reichstein2} or~\cite[Lemma 8.5.7]{ry}.
(If $k$ has a primitive $p$th root of unity then in fact 
equality holds but we won't use this in the sequel.)
For $r = 2$, Corollary~\ref{cor1} is valid but is not optimal. 
Indeed, in this case L.~H.~Rowen and D.~J.~Saltman showed that,
after a prime-to-$p$ extension $L/K$, every degree $p^2$ central 
simple algebra $A/K$ becomes a $(\bbZ/p \bbZ)^2$-crossed product; 
see~\cite[Corollary 1.3]{rs}. 
The upper bound on the essential dimension of a crossed 
product given by \cite[Corollary 3.10]{lrrs} then yields 
the inequality
\[ \ed(\PGL_{p^2}; p) \le p^2 + 1 \, , \]
which is stronger than Corollary~\ref{cor1} for any $p \ge 3$.
If $r \ge 3$ we do not know how close the true value 
of $\ed(\PGL_{p^r}; p)$ is to $\ed(N; p) = p^{2r-1} - p^r + 1$;
in this case Corollary~\ref{cor1} gives the best currently known 
upper bound on $\ed(\PGL_{p^r}; p)$.  We remark that, beyond 
the obvious inequality $\ed(\PGL_{p^r}; p) \le \ed(\PGL_{p^r})$,
the relationship between $\ed(\PGL_{p^r}; p)$ and $\ed(\PGL_{p^r})$ 
is quite mysterious as well.

A key ingredient in our proofs 
of the lower bounds in Theorem~\ref{thm1}(c) and (d) is 
a recent theorem of N.~A.~Karpenko and A.~S.~Merkurjev~\cite{km} 
on the essential dimension of a $p$-group, stated 
as Theorem~\ref{thm.km} below. To the best of our knowledge, 
these results were not accessible by previously existing 
techniques.  Corollary~\ref{cor1} and the other parts 
of Theorem~\ref{thm1} do not rely on the Karpenko-Merkurjev 
theorem.

\section{A general strategy}
\label{sect1.5}

Let $G$ be an algebraic group defined over a field $k$.
Recall that the action of $G$ on an 
algebraic variety $X$ defined over $k$ is generically 
free if the stabilizer subgroup $\Stab_G(x)$ is trivial
for $x \in X(\overline{k})$ in general position.

\begin{remark} \label{rem1.5-1}
If $G$ is a finite constant group and $X$ is irreducible and
smooth then the $G$-action on $X$ is generically free if 
and only if it is faithful. 

Indeed, the ``only if" implication is obvious.
Conversely, if the $G$-action on $X$ is faithful
then $\Stab_G(x) = \{ 1 \}$ for any $x$ outside 
of the closed subvariety 
$\bigcup_{1\ne g \in G} X^{\langle g \rangle}$,
whose dimension is $\le \dim(X)$. 
\qed
\end{remark}

In the course of this paper we will repeatedly encounter the following 
situation.  Suppose we want to show that 
\begin{equation} \label{e.strategy1}
\ed(G) = \ed_k(G; p) = d \, , 
\end{equation}
where $k$ is a field and $G$ is a linear algebraic group defined over $k$.

\smallskip
All such assertions will be proved by the following $2$-step procedure.

\smallskip
(i) Construct a generically free
linear representation of $G$ 
of dimension $d + \dim(G)$. This implies that $\ed_k(G) \le d$; 
see \cite[Theorem 3.4]{reichstein2} or
\cite[Proposition 4.11]{bf}.

\smallskip
(ii) Prove the lower bound $\ed_k(G; p) \ge d$. 

\medskip
Since clearly $\ed(G; p) \le \ed(G)$, the desired 
equality~\eqref{e.strategy1} follows from (i) and (ii).  

\smallskip
The group $G$ will always be of the form $G = D \rtimes F$,
where $D$ is diagonalizable and $F$ is finite. In the next section 
we will recall some known facts about representations of 
such groups. This will help us in carrying out step (i)
and, in the most interesting cases, step (ii) as well, 
via the Karpenko-Merkurjev Theorem~\ref{thm.km}.

\section{Representation-theoretic preliminaries}
\label{sect2}

We will work over a ground field $k$ which remains fixed throughout.
Suppose that a linear algebraic $k$-group $G$ contains
a diagonalizable (over $k$) 
group $D$ and the quotient $G/D$ is a constant finite group $F$. 
Here by ``diagonalizable over $k$" we mean that $D$ is a subgroup 
of the split torus $\bbG_m^d$ defined over $k$ or,
equivalently, that every linear representation of $D$
defined over $k$ decomposes as a direct sum of 1-dimensional 
subrepresentations.

Denote the group of (multiplicative) characters of $D$ by $X(D)$. Note
that since $D$ is diagonalizable over $k$, every multiplicative 
character of $D$ is defined over $k$.
Consider a linear $k$-representation $G \to \GL(V)$.  
Restricting this representation to $D$, 
we decompose $V$ into a direct sum of 1-dimensional character 
spaces.  Let $\Lambda \subset X(D)$ be the set of characters (weights)
of $D$ which occur in this decomposition. 
Note that here $|\Lambda| \le \dim(V)$, and equality holds
if and only if each character from $\Lambda$ occurs in $V$ 
with multiplicity $1$. The finite group  
$F$ acts on $X(D)$ and $\Lambda$ is invariant under this action. 
Moreover, if the $G$-action (and hence, the $D$-action)
on $V$ is generically free then $\Lambda$ generates $X(D)$
as an abelian group. In summary, we have proved the following 
lemma; cf. \cite[Section 8.1]{serre-rep}.

\begin{lem} \label{lem.characters}
Suppose every $F$-invariant generating set $\Lambda$ of $X(D)$ 
contains $\ge d$ elements. If $G \to \GL(V)$ is a generically 
free $k$-representation of $G$ then $\dim(V) \ge d$. 
\qed
\end{lem}

As we explained in the previous section, we are 
interested in constructing low-dimensional generically 
free representations of $G$. In this section we will prove simple
sufficient conditions for generic freeness for two particular 
families of representations.

\begin{lem} \label{lem.gen-free0}
Let $W$ be a faithful representation of $F$ and
$V$ be a representation of $G$ whose restriction to $D$ 
is generically free. Then $V \times W$ is 
a generically free representation of $G$.
\end{lem}

Here we view $W$ as a representation of $G$
via the natural projection $G \to G/D = F$.

\begin{proof} For $w \in W(\overline{k})$ in general position, 
$\Stab_{G}(w) = D$; cf. Remark~\ref{rem1.5-1}.
Choosing $v$ in general position in $V(\overline{k})$, 
we see that
\[ \Stab_{G}(v, w) = \Stab_{G}(v) 
\cap \Stab_{G}(w)= \Stab_D(v) = \{ 1 \} \, . \]
\end{proof}

From now on we will assume that 
$G = D \rtimes F$ is the semidirect product of $D$ and $F$.
In this case, given an $F$-invariant generating set 
$\Lambda \subset X(D)$, we can construct a linear (in fact, a {\em monomial})
$k$-representation $V_{\Lambda}$ of $G$ so that each character
from $\Lambda$ occurs in $V_{\Lambda}$ exactly once. To do this,
we associate a basis element $v_{\lambda}$ to each
$\lambda \in \Lambda$.  The finite group $F$ acts on
\[ V_{\Lambda} = \Span(v_{\lambda} \, | \, \lambda \in \Lambda) \]
by permuting these basis elements in the natural way, i.e., via
\begin{equation} \label{e.P_n}
\sigma \colon v_{\lambda} \mapsto v_{\sigma(\lambda)} \, .
\end{equation}
for any $\sigma \in F$ and any $\lambda \in \Lambda$.
The diagonalizable group $D$-acts by the character $\lambda$
on each 1-dimensional space $\Span(v_{\lambda})$, i.e., via
\begin{equation} \label{e.T}
t \colon v_{\lambda} \mapsto \lambda(t) v_{\lambda}
\end{equation}
for any $t \in D$ and $\lambda \in \Lambda$. Extending
\eqref{e.P_n} and \eqref{e.T} linearly to all of $V_{\Lambda}$, we
obtain a linear representation $G = D \rtimes F \to \GL(V_{\Lambda})$.
Note that by our construction $\dim(V_{\Lambda}) = |\Lambda|$.

Our second criterion for generic freeness is 
a variant of~\cite[Lemma 3.1]{lr} 
or~\cite[Proposition 2.1]{lemire}. For the sake of completeness 
we outline a characteristic-free proof.

\begin{lem} \label{lem.gen-free}
Let $\Lambda$ be an $F$-invariant subset of $X(D)$
and $\phi \colon \bbZ[\Lambda] \to X(D)$ be the natural morphism
of $\bbZ[F]$-modules, taking $\lambda \in \Lambda$ to itself.
Let $V_{\Lambda}$ be the linear representation of
$G = D \rtimes F$ defined by~\eqref{e.P_n} and~\eqref{e.T}, as
above. The $G$-action on $V_{\Lambda}$
is generically free if and only if

\smallskip
(a) $\Lambda$ spans $X(D)$ (or equivalently, $\phi$ is surjective) and

\smallskip
(b) the $F$-action on $\Ker(\phi) $ is faithful.
\end{lem}

\begin{proof} 
Let $U \simeq \bbG_m^n$ be the diagonal subgroup 
of $\GL(V_{\Lambda})$,
in the basis $e_{\lambda}$, where $\lambda \in \Lambda$. 
Here $n = |\Lambda| = \dim(V_{\Lambda})$. The $G$-action on $V$
induces an $F$-equivariant morphism $\rho \colon D \to U$, which
is dual to $\phi$ under the usual (anti-equivalence) $\Diag$
between finitely generated abelian groups and 
diagonalizable algebraic groups.
Applying $\Diag$ to the exact sequence
\[
(0) \rTo \Ker(\phi) \rTo \bbZ[\Lambda] \rTo^\phi  X(D) \rTo \Cok(\phi) 
\rTo (0) \, , \]
of finitely generated abelian $\bbZ[F]$-modules we obtain
an $F$-equivariant exact sequence
\[
1\rTo N\rTo D\rTo^\rho U\rTo Q\rTo1 \, , 
\]
of diagonalizable groups, where
$U= \Diag(\bbZ[\Lambda])$, $N= \Diag(\Cok(\phi))$ 
and $Q= \Diag(\Ker(\phi))$; cf.~\cite[I 5.6]{Ja} or \cite[IV 1.1]{DG}.
Since $U$ is $F$-equivariantly isomorphic to a dense open subset 
of $V$, the $G$-action on $V$ is generically 
free if and only if the $G$-action on $U$ 
is generically free. On the other hand, the $G$-action on
$U$ is generically free if and only if (i) the $D$-action on $U$
is generically free, and (ii) the $F$-action on $Q$ is generically free.

It is now easy to see that (i) is equivalent to (a) and (ii) is
equivalent to (b); cf. Remark~\ref{rem1.5-1}.
\end{proof}

\section{Subgroups of prime-to-$p$ index}
\label{section.reductions}

Our starting point is the following lemma.

\begin{lem} \label{lem.Sylow} Let $G'$ be a closed subgroup 
of a smooth algebraic group $G$ defined over $k$. Assume 
that the index $[G:G']$ is finite and prime to $p$. Then 
$\ed(G; p) = \ed(G'; p)$.
\end{lem}

In the case where $G$ is finite a proof can be found 
in~\cite[Proposition 4.10]{merkurjev}; the argument below
proceeds along similar lines.

\begin{proof}
Recall that if $G$ is a linear algebraic group 
and $H$ is a closed subgroup then
\begin{equation} \label{e.subgroup}
\ed(G; p) \ge \ed(H; p) + \dim(H) - \dim(G) \, ; 
\end{equation}
for any prime $p$; see,~\cite[Lemma 2.2]{brv}
or~\cite[Corollary 4.3]{merkurjev}. 
Since $\dim G^\prime=\dim G$, this 
yields $\ed(G; p)\ge\ed(G^\prime; p)$. 

To prove the opposite inequality, it suffices 
to show that for any field $K/k$ the map 
$H^1(K,G^\prime)\rightarrow H^1(K,G)$ induced 
by the inclusion $G^\prime\subset G$ is $p$-surjective, 
i.e., that for every $\alpha\in H^1(K,G)$ there is 
a finite field extension $L/K$ of degree prime to $p$ 
such that $\alpha_L$ is in the image of 
$H^1(L,G^\prime)\rightarrow H^1(L,G)$;
see, e.g.,~\cite[Proposition 1.3]{merkurjev}.

Let $X$ be a $G$-torsor over $K$ and $X/G$ 
the quotient by the action of $G^\prime$. 
For a field $L/K$ and an $L$-point $\Spec(L)\rightarrow X/G^\prime$ 
we construct a $G^\prime$-torsor $Y$ as the pullback 
\[
\begin{diagram}[size=0.6cm]
Y&\rTo&X\\
\dTo&&\dTo\\
\Spec(L)&\rTo&X/G^\prime\\
&&\dTo\\
&& \Spec(K)
\end{diagram}
\]
In this situation $Y\times^{G^\prime} G \cong X_L$ 
as $G$-torsors. Thus we have the natural diagram
\[     
\begin{diagram}[size=0.6cm]
H^1(L,G^\prime)&\rTo&&&H^1(L,G)\\
&[Y]&\rMapsto&[X]_L&\uTo\\
&&&\uMapsto&\\
&& &[X]&&\\
&&&&H^1(K,G)
\end{diagram}
\]    
where $[X]$ and $[Y]$ denote the classes of $X$ and $Y$ in $H^1(K, G)$
and $H^1(L, G')$, respectively. 
It remains to show the existence of such an $L$-point, 
with the degree $[L:K]$ prime to $p$.

Note that $G/G^\prime$ is affine, since $G$ and $G^\prime$ are of the same dimension and hence 
$G/G^\prime\cong (G/G^\circ)/(G^\prime/G^\circ)=\Spec k[G/G^\circ]^{G^\prime/G^\circ}$ where $G^\circ$ is the
connected component of $G$ (and $G^\prime$). 
Furthermore $G/G^\prime$ is smooth; cf. \cite[III 3.2.7]{DG}.
Let $K_s$ be the separable closure of $K$. 
$X$ being a $G$-torsor, we 
have $X_{K_s}\cong G_{K_s}$ 
and $(X/G^\prime)_{K_s}\cong (G/G^\prime)_{K_s}$ which implies that $X/G^\prime$
is also affine, cf. \cite[III 3.5.6 d)]{DG}. 
Thus, $K[X/G^\prime]\otimes K_s\cong 
k[G/G^\prime]\otimes K_s$ is reduced and its dimension 
$\dim_K K[X/G^\prime]=[G:G^\prime]$ is not divisible by $p$ by assumption. 

Therefore $K[X/G^\prime]$ is \'etale or, equivalently, 
a product of separable field extensions of $K$
\[
K[X/G^\prime]=L_1\times\cdots\times L_r;
\]
see, e.g.,~\cite[V, Theorem 4]{bourbaki}.
For each $L_j$ the projection 
$K[X/G^\prime]\rightarrow L_j$ is an $L_j$-point 
of $X/G'$
and since \[
\dim_K K[X/G^\prime]=\sum_{j=1}^r [L_j:K]\quad\mbox{ is prime to $p$,}
\]
one of the fields $L_j$ must be of degree prime to $p$ over $K$. 
We now take $L=L_j$.
\end{proof}

\begin{cor} \label{cor.S_n} Suppose $k$ is a field of 
characteristic $\ne p$ containing a primitive $p$th 
root of unity. Then $\ed_k(\Sym_n; p) = [n/p]$.
\end{cor}

\begin{proof}
Let $m = [n/p]$ and let $D \simeq (\bbZ/p \bbZ)^m$ be the 
subgroup generated by the disjoint $p$-cycles 
\[ \sigma_1 = (1, \dots, p), \ldots, 
 \sigma_m = ((m-1)p + 1, \dots, mp) \, . \]
The inequality $\ed(\Sym_n; p) \ge \ed_k(D; p) \ge  [n/p]$ is well known; 
see,~\cite[Section 6]{br1},~\cite[Section 7]{br2},
or~\cite[Proposition 3.7]{bf}.

To the best of our knowledge, the opposite inequality
was first noticed by J.-P. Serre (private communication, 
May 2005) and independently by R. L\"otscher~\cite{loetscher}.
The proof is quite easy; however, since it has not previously 
appeared in print, we reproduce it below.

The semi-direct product $D \rtimes \Sym_m$, where
$\Sym_m$ permutes $\sigma_1, \dots, \sigma_m$, 
embeds in $\Sym_n$ with index prime to $p$.
By Lemma~\ref{lem.Sylow},
$\ed_k(D \rtimes \Sym_m ;p) = \ed_k(\Sym_n;p)$,
and it suffices to show that $\ed_k(D \rtimes \Sym_m) \le [n/p]$.

As we mentioned in Section~\ref{sect1.5}, in order to prove this,
it suffices to construct a generically free 
$m$-dimensional representation of $D \rtimes \Sym_m$
defined over $k$. To construct such a representation,
let $\sigma_1^*, \dots, \sigma_m^* \subset X(D)$ 
be the ``basis" of $D$ dual to $\sigma_1, \dots, \sigma_m$.  
That is, we choose a primitive $p$th 
root of unity $\zeta \in k$ and set
\[ \sigma_i^*(\sigma_j) = \begin{cases} \text{$\zeta$, if $i = j$ and} \\
\text{$1$, otherwise.} \end{cases} \]
The $\Sym_m$-invariant subset $\Lambda = \{
\sigma_1^*, \dots, \sigma_m^* \}$ of $X(D)$
gives rise to the $m$-dimensional $k$-representation 
$V_{\Lambda}$ of $D \rtimes \Sym_m$, as 
in Section~\ref{sect2}.  An easy 
application of Lemma~\ref{lem.gen-free} shows that 
this representation is generically free.
\end{proof}

\section{First reductions and proof of Theorem~\ref{thm1} parts {\rm(a)} and {\rm(b)}} 

Let $T \simeq \bbG_m^n/\Delta$ be the diagonal 
maximal torus in $\PGL_n$, where $\Delta = \bbG_m$
is diagonally embedded into $\bbG_m^n$. Recall that
the normalizer $N$ of $T$ is isomorphic to
$T \rtimes \Sym_n$, where we identify $\Sym_n$ with the
subgroup of permutation matrices in $\PGLn$.

Let $P_n$ be a Sylow $p$-subgroup of $\Sym_n$.
Lemma~\ref{lem.Sylow} tells us that
\[ \ed(N; p) = \ed(T \rtimes P_n; p) \, . \] 
Thus in order to prove Theorem~\ref{thm1} it suffices to establish
the following proposition.

\begin{prop} \label{prop1}
Let $T \simeq \bbG_m^n/\Delta$, where
$\Delta = \bbG_m$ is diagonally embedded into $\bbG_m^n$. 
Assume that a field $k$ is of characteristic $\ne p$ and
containing a primitive $p$th root of unity.  Then

\smallskip
\begin{tabular}{lll}
(a) $\ed_k(T \rtimes P_n) = \ed_k(T \rtimes P_n; p) = [n/p]$, & 
if $n$ is not divisible by $p$. \\
(b)  $ \ed_k(T \rtimes P_n) = \ed_k (T \rtimes P_n; p) = 2$, 
& if $n=p$. \\
(c)  $\ed_k(T \rtimes P_n) = \ed_k (T \rtimes P_n; p) = n^2/p - n + 1$, 
& if $n=p^r$ for some $r\ge 2$. \\
(d)  $\ed_k( T \rtimes P_n) = \ed_k (T \rtimes P_n; p) = p^e(n-p^e)-n+1$, & in all other cases.
\end{tabular}
            
\smallskip  
\noindent
Here $P_n$ is a Sylow $p$-subgroup of $\Sym_n$,
$[n/p]$ is the integer part of $n/p$ and
$p^e$ is the highest power of $p$ dividing $n$.
\end{prop}   

Our proof of each part of this proposition will be based on the strategy
outlined in Section~\ref{sect1.5}, with $G = T \rtimes P_n$.
Before we proceed with the details, we recall that
the character lattice $X(T)$ is naturally isomorphic to
\[ \{ (a_1, \dots, a_n) \in \bbZ^n \, | \, a_1 + \dots + a_n = 0 \} \, , \]
where we identify the character
\[ (t_1, \dots, t_n) \to t_1^{a_1} \ldots t_n^{a_n} \]
of $T = \bbG_m^n/\Delta$
with $(a_1, \dots, a_n) \in \bbZ^n$. Note that $(t_1, \dots, t_n)$
is viewed as an element of $\bbG_m^n$ modulo the diagonal subgroup    
$\Delta$, so the above character is well defined if and only
if $a_1 + \dots + a_n = 0$.                                 
An element $\sigma$ of $\Sym_n$ (and in particular,         
of $P_n \subset \Sym_n$) acts on ${\bf a} = (a_1, \dots, a_n) \in X(T)$
by naturally permuting $a_1, \dots, a_n$.                        

For notational convenience, we will denote by ${\bf a}_{i, j}$ 
the element of $(a_1, \dots, a_n) \in X(T)$ such that
$a_i = 1$, $a_j = -1$ and $a_h = 0$ for every $h \ne i, j$. 

We also recall that fro $n = p^r$ the Sylow $p$-subgroup $P_n$ 
of $\Sym_n$ can be described inductively as the wreath product 
\[ P_{p^r}\cong P_{p^{r-1}}\wr \bbZ/p\cong (P_{p^{r-1}})^p\rtimes \bbZ/p 
\, . \]
For general $n$, $P_n$ is the direct product of certain $P_{p^r}$, 
see Section~\ref{sect.generaln}.

\begin{proof}[Proof of Proposition~\ref{prop1}(a)]

\smallskip
Step (i): Since $n$ is not divisible by $p$, we may assume that
$P_n$ is contained in $\Sym_{n-1}$, where we identify $S_{n-1}$
with the subgroup of $\Sym_n$ consisting of permutations 
$\sigma \in \Sym_n$ such that $\sigma(1) = 1$. 

We will now construct a generically free linear representation 
$V$ of $T \rtimes S_{n-1}$ of dimension $n - 1 + [n/p]$.
Restricting this representation to $T \rtimes P_n$, we will obtain
a generically free linear representation of 
dimension $n - 1 + [n/p]$. This will show that $\ed(T \rtimes P_n) \le [n/p]$.

To construct $V$, let $\Lambda = \{ {\bf a}_{1, i} \, | \, i = 2, \dots, n \}$
and let $W$ be a $[n/p]$-dimensional faithful linear representation 
of $P_n$ constructed in the proof of Corollary~\ref{cor.S_n}. Applying 
Lemma~\ref{lem.gen-free0}(b), we see that $V = V_{\Lambda} \times W$ 
is generically free. 

\smallskip
Step (ii): Since the natural projection $p \colon T \rtimes P_n \to P_n$ 
has a section,
so does the map $p^* \colon H^1(K, T \rtimes P_n) \to H^1(K, P_n)$ 
of Galois cohomology sets. Hence, $p^*$ is surjective for every 
field $K/k$. This implies that 
\[ \ed(T \rtimes P_n) \ge \ed(P_n; p) = [n/p] \, . \]
(Note that by Lemma~\ref{lem.Sylow} $\ed(P_n; p) = \ed(\Sym_n; p)$;
and by Corollary~\ref{cor.S_n} $\ed(\Sym_n; p) = [n/p]$.)
\end{proof}

\begin{remark} \label{rem.proof2}
We will now outline a different (and perhaps, more conceptual)
proof of the upper bound $\ed(N; p) \le [n/p]$
of Theorem~\ref{thm1}(a). As we pointed out in the introduction,
$\ed(N; p)$ is the essential dimension at $p$ of the functor
\[ \text{$H^1( \, \ast \, , N) \colon K \mapsto \{$ $K$-isomorphism classes 
of pairs $(A, L)$ $\}$,} \]
where $A$ is a degree $n$ central simple algebras over $K$, $L$
is a maximal \'etale subalgebra of $A$.
Similarly, $\ed(\Sym_n; p)$ is the essential dimension at $p$ of the functor
\[ \text{$H^1( \, \ast \, , \Sym_n) \colon K \mapsto 
\{ K$-isomorphism classes of $n$-dimensional \'etale algebras $L/K \, \}$.} \]
Let $\alpha \colon 
H^1( \, \ast \, , \Sym_n) \to H^1( \, \ast \, , N)  $
be the map taking an $n$-dimensional \'etale algebra $L/K$ to $(\End_K(L), L)$.
Here we embed $L$ in $\End_K(L) \simeq \Mat_n(K)$ via the regular 
action of $L$ on itself.  

It is easy to see that, in the terminology of~\cite[Section 1.3]{merkurjev},
$\alpha$ is $p$-surjective. That is, for any class
$(A, L)$ in $H^1(K, N)$ there exists a prime-to-$p$ 
extension $K'/K$ such that $(A \otimes_K K', L \otimes_K K')$ 
lies in the image of 
$\alpha$. In fact, any $K'/K$ of degree prime-to-$p$
which splits $A$ will do; indeed,
by the Skolem-Noether theorem, any two embeddings of $L \otimes_K K'$
into $\Mat_n(K')$ are conjugate. By \cite[Proposition 1.3]{merkurjev},
we conclude that $\ed(N; p) \ge \ed(\Sym_n; p)$. Combining this with
Corollary~\ref{cor.S_n} yields the desired inequality
$\ed(N; p) \le [n/p]$.
\qed
\end{remark}

\begin{proof}[Proof of Proposition~\ref{prop1}(b)]
Here $n = p$ and $P_n \simeq \bbZ/p$ is generated by the $p$-cycle
$(1, 2, \dots, n)$.
We follow the strategy outlined in Section~\ref{sect1.5}.

\smallskip
Step (i): To show that $\ed_k(T \rtimes P_n) \le 2$, we will
construct a generically free $k$-representation of $T \rtimes P_n$
of dimension $2 + \dim(T \rtimes P_n) = n + 1$.

Let $\Lambda = \{ {\bf a}_{1, 2}, \ldots, 
{\bf a}_{p-1, p}, {\bf a}_{p, 1} \}$ and $V = V_{\Lambda} \times L$,
where $L$ is a $1$-dimensional faithful representation of
$P_n \simeq \bbZ/p$ and $T \rtimes P_n$ acts on $L$ via the natural 
projection $T \rtimes P_n \to P_n$. 
Note that $\dim(V) = | \Lambda| + 1 = n + 1$.
Since $\Lambda$ generates $X(T)$, Lemma~\ref{lem.gen-free0}(b) tells us that 
$V$ is a generically free representation of $T \rtimes P_n$.

\smallskip
Step (ii):  Recall that $\ed_k(T \rtimes P_n; p) = \ed_k(N; p)$
by Lemma~\ref{lem.Sylow}. On the other hand, as we mentioned 
in the introduction, 
\[ \ed_k(N; p) \ge \ed_k(\PGL_p; p) \ge 2 \, ; \]
see~\eqref{e.N} and~\eqref{e.pgl_p}. 
This completes the proof of Proposition~\ref{prop1}(b) 
and of Theorem~\ref{thm1}(b).
\end{proof}

\section{Proof of Theorem~\ref{thm1} part {\rm(c)}: 
The upper bound}\label{sect.upperbound}

In the next two sections we will prove Proposition~\ref{prop1}(c)
and hence, Theorem~\ref{thm1}(c). We will assume that 
$n=p^r$ for some $r\ge 2$ and follow the strategy 
of Section~\ref{sect1.5}. In this section we will carry out Step (i).
That is, we will construct construct a generically 
free representation $V$ of $T\rtimes P_n$ of dimension $p^{2r-1}$. 
Our $V$ will be of the form $V_{\Lambda}$ for a particular
$P_n$-invariant $\Lambda \subset X(T)$, following the recipe
of Section~\ref{sect2}.

For notational convenience, 
we will subdivide the integers $1, 2, \dots, p^r$ into $p$ ``big blocks"
$B_1, \dots, B_p$, where each $B_i$ consists of the $p^{r-1}$ integers   
$(i-1)p^{r-1} + 1, (i-1)p^{r-1} + 2, \dots, ip^{r-1}$. 

We define $\Lambda \subset X(T)$ as the $P_n$-orbit of the element
\[ {\bf a}_{1, p^{r-1} + 1} = (\underbrace{1, 0, \dots, 0}_{B_1}, 
\underbrace{-1, 0, \dots, 0}_{B_2}, 
\underbrace{0, 0, \dots, 0}_{B_3}, \ldots, 
\underbrace{0, 0, \dots, 0}_{B_p}) \] 
in $X(T)$. 
Thus, $\Lambda$ consists of elements
${\bf a}_{\alpha, \beta}$, subject to the condition that
if $\alpha$ lies in the big block $B_i$ then $\beta$ has to lie in 
$B_j$, where $j - i \equiv 1$ modulo $p$.
There are $p^r$ choices for $\alpha$. Once $\alpha$ is chosen,
there are exactly $p^{r-1}$ further choices for $\beta$.
Thus \[ |\Lambda| = p^r \cdot p^{r -1} = p^{2r -1} \, . \]
As described in Section~\ref{sect2}, we obtain
a linear representation $V_{\Lambda}$
of $T \rtimes P_n$ of the desired dimension 
\[ \dim(V_{\Lambda}) = |\Lambda| = p^{2r-1} \, . \]
It remains to prove that $V_{\Lambda}$ is generically free.
By Lemma~\ref{lem.gen-free} it suffices to show that 

\smallskip
(i) $\Lambda$ generates $X(T)$ as an abelian group and

\smallskip
(ii) the $P_n$ action on the kernel
of the natural morphism $\phi \colon \bbZ[\Lambda] \to X(T)$ 
is faithful.  

\smallskip
The elements ${\bf a}_{\alpha, \beta}$ clearly generate $X(T)$ 
as an abelian group, as $\alpha$ and $\beta$ range over 
$1, 2, \dots, p^r$. Thus in order to prove (i)
it suffices to show that $\Span_{\bbZ}(\Lambda)$
contains every element of this form. Suppose $\alpha$ lies 
in the big block $B_i$ and $\beta$ in $B_j$. If $j - i \equiv 1 \pmod{p}$, then
${\bf a}_{\alpha, \beta}$ lies in $\Lambda$ and there is nothing 
to prove. If $j - i \equiv 2 \pmod{p}$ then choose some 
$\gamma \in B_{i + 1}$ 
(where the subscript $i + 1$ should be viewed modulo $p$) and write
\[ {\bf a}_{\alpha, \beta} = {\bf a}_{\alpha, \gamma} +
 {\bf a}_{\gamma, \beta} \, . \] 
Since both terms on the right are in $\Lambda$, we see that
in this case ${\bf a}_{\alpha, \beta} \in \Span_{\bbZ}(\Lambda)$.
Using this argument recursively, we see that  
${\bf a}_{\alpha, \beta}$ also lies in $\Span_{\bbZ}(\Lambda)$
if $j - i \equiv 3, \dots, p \pmod{p}$, i.e., for all possible $i$ and $j$.
This proves (i).

To prove (ii), denote the kernel of $\phi$ by $M$. Since $P_n$ is 
a finite $p$-group, every normal subgroup of $P_n$ intersects 
the center of $P_n$, which we shall denote by $Z_n$. Thus 
it suffices to show that $Z_n$ acts faithfully on $M$.

Recall that $Z_n$ is the cyclic subgroup of $P_n$ of order $p$ generated
by the product of disjoint $p$-cycles 
\[
\sigma_1 \cdot \ldots \cdot \sigma_{p^{r-1}} = (1 \dots p)(p+1 \dots 2p)
\ldots (p^r - p + 1, \dots, p^r) \, . 
\]
Since $|Z_n| = p$, it either acts faithfully on $M$ or it acts 
trivially, so we only need to check that the $Z_n$-action on $M$ 
is non-trivial.
Indeed,
 $Z_n$ does not fix the non-zero element
\[
{\bf a}_{1, p^{r-1} + 1} + {\bf a}_{p^{r-1} + 1, 2 p^{r-1} + 1} + 
\dots + 
{\bf a}_{(p-1) p^{r-1} + 1, 1} \in \bbZ[\Lambda]
\]
which lies in $M$. 
This completes the proof of the upper bound of Proposition~\ref{prop1} and Theorem~\ref{thm1}(c).
\qed

\section{Theorem~\ref{thm1} part {\rm(c)}: The lower bound}
\label{section.lowerbound}

In this section we will continue to assume that $n=p^r$. We will
show that 
\begin{equation} \label{e.lowerbound(c)}
\ed(N;p) \ge p^{2r-1}-p^r+1 \, ,
\end{equation}
thus completing the proof 
of Proposition~\ref{prop1}(c) and Theorem~\ref{thm1}(c). 

First we remark that $\ed_k(G) \ge \ed_{\overline{k}}(G)$,
where $\overline{k}$ is the algebraic closure of $k$;
cf., e.g.,~\cite[Proposition 1.5]{bf}. 
Thus, for the purpose of proving the lower bound~\eqref{e.lowerbound(c)}
we may replace $k$ by $\overline{k}$
and assume that $k$ is algebraically closed.  Let
\begin{equation} \label{e.q}
\text{$q := p^e$, where $e \ge 1$ if $p$ is odd and $e \ge 2$ if $p = 2$.}
\end{equation}
be a power of $p$.  The specific choice of $e$ will not be important 
in the sequel; in particular, the reader may assume that $q = p$ if $p$ is 
odd and $q = 4$, if $p = 2$. Whatever $q$ we choose 
(subject to the above constraint) it will 
remain unchanged for the rest of this section.

Let $T_{(q)} = \mu_q^n/\mu_q$ be 
the $q$-torsion subgroup of $T = \bbG_m^n/\Delta$.  Applying 
the inequality~\eqref{e.subgroup} to $G = T \rtimes P_n$ and
its finite subgroup $H = T_{(q)} \rtimes P_n$, 
we obtain
\[ \ed(T \rtimes P_n;p) \ge \ed(T_{(q)} \rtimes P_n; p) - p^r + 1  \, . \]
Thus 
it suffices to show that 
\begin{equation} \label{ed.finite}
\ed(T_{(q)} \rtimes P_n; p) \ge p^{2r - 1} \, . 
\end{equation}
The advantage of replacing $T \rtimes P_n$ by $T_{(q)} \rtimes P_n$ is 
that $T_{(q)} \rtimes P_n$ is a finite $p$-group, so that we can
apply the following recent result of Karpenko and Merkurjev~\cite{km}.

\begin{thm} \label{thm.km} Let $G$ be a finite $p$-group and
$k$ be a field containing a primitive $p$th root of unity.
Then $\ed_k(G; p) = \ed_k(G)$ = the minimal value of
$\dim(V)$, where $V$ ranges over all faithful linear
$k$-representations $G \to \GL(V)$.
\end{thm}

Since we are assuming that $k$ is algebraically closed,
it thus remains to show that $T_{(q)} \rtimes P_n$
does not have a faithful linear representation of dimension 
$< p^{2r-1}$. Lemma~\ref{lem.characters} further reduces 
this representation-theoretic assertion to the combinatorial
statement of Proposition~\ref{prop.lower-bound} below. 

Before stating Proposition~\ref{prop.lower-bound} we recall that 
the character lattice of $T_{(q)}$ is 
\[ \text{$X_n:= \{ (a_1, \dots, a_n) \in (\bbZ/q\bbZ)^n\, | \, 
a_1 + \dots + a_n = 0$ in $\bbZ/q \bbZ \; \}$,}  \]
where we identify the character 
\[ (t_1, \dots, t_n) \to t_1^{a_1} \dots t_n^{a_n} \]
of $T_{(q)}$ with $(a_1, \dots, a_n) \in (\bbZ/q \bbZ)^n$.
Here $(t_1, \dots, t_n)$ stands for an element of
$\mu_q^n$, modulo the diagonally embedded $\mu_q$, so
the above character is well defined if and only 
if $a_1 + \dots + a_n = 0$ in $\bbZ/q \bbZ$. (This is completely 
analogous to our description of the character lattice of $T$ 
in the previous section.) Note that $X_n$ depends on the integer $q = p^e$, which 
we assume to be fixed throughout this section.

\begin{prop} \label{prop.lower-bound} 
Let $n = p^r$ and $P_n$ be a Sylow $p$-subgroup of $\Sym_n$. 
If $\Lambda$ is a $P_n$-invariant generating subset of $X_n$ 
then $|\Lambda| \ge p^{2r-1}$ for any $r \ge 1$.
\end{prop}

Our proof of Proposition~\ref{prop.lower-bound} will rely 
on the following special case of Nakayama's 
Lemma~\cite[Proposition 2.8]{am}.

\begin{lem} \label{lem.lin-algebra}
Let $q = p^e$ be a prime power, $M = (\bbZ/q \bbZ)^d$ and
$\Lambda$ be a generating subset of $M$ (as an abelian group).
If we remove from $\Lambda$ all elements that lie in $pM$,
the remaining set, $\Lambda \setminus pM$, will still generate $M$. 
\qed
\end{lem}

\begin{proof}[Proof of Proposition~\ref{prop.lower-bound}] 
We argue by induction on $r$. For the base case, set $r = 1$.  We need 
to show that $|\Lambda| \ge p$. Assume the contrary.  In this case 
$P_n$ is a cyclic $p$-group, and every non-trivial orbit of $P_n$ 
has exactly $p$ elements. Hence, $|\Lambda| < p$ is only possible if
every element of $\Lambda$ is fixed by $P_n$. Since we are assuming that  
$\Lambda$ generates $X_n$ as an abelian group, we conclude that $P_n$
acts trivially on $X_n$. This can happen only if $p = q = 2$. Since
these values are ruled out by our definition~\eqref{e.q} of $q$,
we have proved the proposition for $r = 1$.

In the previous section we subdivided
the integers $1, 2, \dots, p^r$ into $p$ ``big blocks" $B_1, \dots, B_p^{r-1}$
of length $p$.  Now we will now work with ``small blocks"
$b_1, \dots, b_{p^{r-1}}$, where $b_j$ consists of the $p$ 
consecutive integers \[ (j-1)p + 1, (j-1)p + 2, \dots, jp \, . \] 
We can identify $P_{p^{r-1}}$ with the subgroup of $P_{p^r}$ 
that permutes the small blocks $b_1, \dots, b_{p^{r-1}}$ without 
changing the order of the elements in each block.

For the induction step, assume $r \ge 2$ and consider the homomorphism
$\Sigma \colon X_{p^r} \to X_{p^{r-1}}$ given by 
\begin{equation} \label{e.Sigma}
{\bf a} = (a_1, a_2, \dots, a_{p^r}) \mapsto 
{\bf s} = (s_1, \dots, s_{p^{r-1}}) \, ,  
\end{equation}
where 
$s_i = a_{(i-1)p + 1} + a_{(i-1)p + 2} + \ldots + a_{ip}$ 
is the sum of the entries of ${\bf a}$ in the $i$th small block $b_i$.
Thus

\smallskip
(i) if $\Lambda$ generates $X_{p^r}$ then $\Sigma(\Lambda)$ generates
$X_{p^{r-1}}$.

\smallskip
(ii) if $\Lambda$ is a $P_{p^r}$-invariant subset of $X_{p^r}$ then 
$\Sigma(\Lambda)$ is a $P_{p^{r-1}}$-invariant subset of $X_{p^{r-1}}$.

\smallskip
Let us remove from $\Sigma(\Lambda)$ all elements 
which lie in $p X_{p^{r-1}}$.
The resulting set,
$\Sigma(\Lambda) \setminus p X_{p^{r-1}}$, is
clearly $P_{p^{r-1}}$-invariant. 
By Lemma~\ref{lem.lin-algebra} this set
generates $X_{p^{r-1}}$.  Thus by the induction 
assumption $|\Sigma(\Lambda) \setminus pX_{p^{r-1}}| 
\ge p^{2r-3}$.

We claim that the fiber of each element
${\bf s} = (s_1, \dots, s_{p^{r-1}})$ in $\Sigma(\Lambda) \setminus
pX_{p^{r-1}}$
has at least $p^2$ elements in $\Lambda$. If we can show this, 
then we will be able to conclude that
\[ |\Lambda| \ge p^2 \cdot | \Sigma(\Lambda) 
\setminus pX_{p^{r-1}}| \ge p^2 \cdot p^{2r-3}= p^{2r -1} \, , \]
thus completing the proof of Proposition~\ref{prop.lower-bound}.

Let $\sigma_i$ be the single $p$-cycle, cyclically permuting
the elements in the small block $b_i$.
To prove the claim, note that the subgroup  
\[ \langle \sigma_i \, | \,  i = 1, \dots, p^{r-1} \rangle 
\simeq (\bbZ/p \bbZ)^{p^{r-1}} \]
of $P_n$ acts on each fiber of $\Sigma$. 

To simplify the exposition in the argument to follow, we
introduce the following bit of terminology.
Let us say that ${\bf a} \in (\bbZ/q \bbZ)^n$ 
is scalar in the small block $b_i$  
if all the entries of ${\bf a}$ in the block $b_i$ are the same, i.e.,
if
\[ a_{(i-1)p + 1} = a_{(i-1)p + 2} = \dots = a_{ip} \, . \]

We are now ready to prove the claim. Suppose ${\bf a} = (a_1, \dots, a_{p^r})
\in X_{p^r}$ lies in the preimage of ${\bf s} = (s_1, \dots, s_{p^{r-1}})$,
as in~\eqref{e.Sigma}. 
If ${\bf a}$ is scalar in the small block $b_i$ then clearly 
\[ s_i = a_{(i-1)p + 1} + a_{(i-1)p + 2} + \dots + a_{ip} \in 
p \bbZ/q \bbZ \, . \] 
Since we are assuming that ${\bf s}$ lies in 
\[ \Sigma(\Lambda) \setminus p X_{p^{r-1}} \, , \]
${\bf s}$ must have at least two entries that are 
not divisible by $p$, say, $s_i$ and $s_j$.
(Recall that $s_1 + \dots + s_{p^r} = 0$ in $\bbZ/q \bbZ$, so 
${\bf s}$ cannot have exactly one entry not divisible by $p$.) Thus
${\bf a}$ is non-scalar in the small blocks $b_i$ and $b_j$. 
Consequently, the elements $\sigma_i^{\alpha}
\sigma_j^{\beta}(a)$ are distinct, as $\alpha$ and $\beta$ range between
$0$ and $p-1$. All of these elements lie in the fiber of ${\bf s}$ 
under $\Sigma$. Therefore we conclude that this fiber contains 
at least $p^2$ distinct elements. This completes 
the proof of the claim and thus of Proposition~\ref{prop.lower-bound},
Proposition~\ref{prop1}(c) and Theorem~\ref{thm1}(c). 
\end{proof}

\section{Proof of Theorem~\ref{thm1} part {\rm(d)}}
\label{sect.generaln}

In this section we assume that $n$ is divisible by $p$ but 
is not a power of $p$. 
We will modify the arguments of the last two sections
to show that 
\[ \ed(T \rtimes P_n) = \ed(T \rtimes P_n; p) = p^e(n-p^e)-n+1 \, ,
 \]
where $p^e$ is the highest power of $p$ dividing $n$.
This will complete the proof of Proposition~\ref{prop1} and thus of 
Theorem~\ref{thm1}.

Write out the $p$-adic expansion 
\begin{equation} \label{e.pad}
n=n_1p^{e_1}+n_2p^{e_2}+...+n_up^{e_u},
\end{equation}
of $n$, where $1 \le e=e_1 < e_2<...<e_u$, and $1\le n_i<p$ for each $i$.
Subdivide the integers $1,...,n$ into $n_1+...+n_u$ 
blocks $B_j^i$ of length $p^{e_i}$, for $j$ ranging over $1, 2, ...,n_i$.
By our assumption there are at least two such blocks.
The Sylow subgroup $P_n$ is a direct product
\[
P_n=(P_{p^{e_1}})^{n_1}\times\cdots \times (P_{p^{e_u}})^{n_u}
\]
where each $P_{p^{e_i}}$ acts on one of the blocks $B_j^i$.

Once again we will use the strategy outlined in Section~\ref{sect1.5}.

\smallskip
Step (i): We will construct a generically free representation 
of $T \rtimes P_n$ of dimension $p^{e_1}(n- p^{e_1})$. This will 
prove the upper bound $\ed(T \rtimes P_n) \le p^{e_1}(n-p^{e_1})$.

To construct this representation,
let $\Lambda \subset X(T)$ be the union of the $P_n$-orbits of the elements
\[ 
{\bf a}_{1, j + 1} \mbox{ where } j=p^{e_1},...,n_1p^{e_1},n_1p^{e_1}+p^{e_2},.....,n-p^{e_u}
\] 
i.e., the union of the $P_n$-orbits of elements
of the form $(1, 0 \ldots, 0, -1, 0, \ldots, 0)$, where $1$
appears in the first position of the first block and 
$-1$ appears in the first position of one of the other blocks.
For ${\bf a}_{\alpha, \beta}$ in $\Lambda$ there are $p^{e_1}$ 
choices for $\alpha$ and 
$n-p^{e_1}$ choices for $\beta$. 
Thus \[ \dim(V_{\Lambda}) = | \Lambda | = p^{e_1}(n-p^{e_1}) \, . \]
It is not difficult to see that $\Lambda$ generates $X(T)$ 
as an abelian group. To conclude with Lemma~\ref{lem.gen-free} 
that $V_{\Lambda}$ is a generically free representation of
$T \rtimes P_n$, it remains to show that the $P_n$ action on the kernel
of the natural morphism $\phi \colon \bbZ[\Lambda] \to X(T)$ 
is faithful when $e_1\ge 1$.
As in section~\ref{sect.upperbound} we only need to check that the center $Z_n$ of $P_n$ acts faithfully on the kernel.
Let $\sigma$ be a non trivial element of $Z_n=(Z_{p^{e_1}})^{n_1}\times\cdots\times(Z_{p^{e_u}})^{n_u}$. 
We may assume that the first component of $\sigma$
in the above direct product is non-trivial, 
and therefore $\sigma$
permutes elements in the first block $B_1^1$ cyclically. Note that
$B_1^1$ is of size at least $p$ as $e=e_1\ge 1$, and that we have 
at least $2$ blocks. The second block is also of 
size $\ge p$ and if $p=2$, at least of size $4$ by \eqref{e.pad}.
It follows from this that $\sigma$ does not fix the non-zero element
\[
{\bf a}_{1, p^e + 1} - {\bf a}_{1, p^e + 2} +
{\bf a}_{2, p^e + 2} - {\bf a}_{2, p^e + 1} 
\]
which lies in the kernel of $\phi$.

\smallskip
Step (ii): We now want to prove the lower bound, 
$\ed(T \rtimes P_n; p) \ge p^{e_1}(n - p^{e_1}) - n + 1$.
Arguing as in Section~\ref{section.lowerbound} (and using 
the same notation, with $q = p$), it suffices to show that 
$\ed(T_{(p)} \rtimes P_n; p) \ge p^{e_1}(n - p^{e_1})$.
By the Karpenko-Merkurjev theorem~\ref{thm.km} this is equivalent 
to showing that every faithful representation of 
$T_{(p)} \rtimes P_n$ has dimension $\ge p^{e_1}(n - p^{e_1})$.
By Lemma~\ref{lem.characters} it now suffices to prove 
the following lemma.

\begin{lem} \label{lem.d(ii)} 
Let $n$ be a positive integer, 
$P_n$ be the Sylow subgroup of $\Sym_n$,
$p^e$ be the highest power of $p$ dividing $n$, and  
\[ \text{$X_n := \{ (a_1, \dots, a_n) \in (\bbZ/p\bbZ)^n\, | \, 
a_1 + \dots + a_n = 0$ in $\bbZ/p \bbZ \; \}$.}  \]
Then every $P_n$-invariant generating subset of $X_n$ 
has at least $p^e(n-p^e)$ elements.
\end{lem}

In the statement of the lemma we allow $e = 0$, to facilitate 
the induction argument.  For the purpose 
of proving the lower bound in Proposition~\ref{prop1}(d)
we only need this lemma for $e \ge 1$.

\begin{proof}
Once again, we consider the $p$-adic expansion \eqref{e.pad} of $n$, with 
$0 \le e_1<e_2<...<e_u$ and  $1\le n_i<p$. We may assume that
$n$ is not a power of $p$, since otherwise the lemma is vacuous.

We will argue by induction on $e = e_1$. 
For the base case, let $e_1=0$. Here the lemma is obvious: 
since $X_n$ has rank $n - 1$, every generating set ($P_n$-invariant 
or not) has to have at least $n - 1$ elements.

For the induction step, we may suppose $e = e_1 \ge 1$; in particular,
$n$ is divisible by $p$. 
Define $\Sigma: X_n \rightarrow X_{n/p}$ 
by sending $(a_1,....,a_n)$ to 
$(s_1,...,s_{n/p})$, where 
\[ s_j=a_{(j-1)p + 1}  + \dots + a_{jp} \]
for $j = 1, \dots, n/p$.  Arguing as in Section~\ref{section.lowerbound} 
we see that $\Sigma(\Lambda) \setminus p X_{n/p}$ is
a $(P_{p^{e_1-1}})^{n_1} \times \dots \times (P_{p^{e_u -1}})^{n_u}$-invariant
generating subset of $X_{n/p}$ and that every
\[ {\bf s} \in \Sigma(\Lambda) \setminus p X_{n/p} \]
has at least $p^2$ preimages in $\Lambda$. By the induction assumption,
\[ | \Sigma(\Lambda) \setminus p X_{n/p} | \ge p^{e-1}(\frac{n}{p} - p^{e - 1})
\]
and thus
\[
| \Lambda | \ge p^2 \cdot p^{e-1}(\frac{n}{p} - p^{e - 1}) = p^{e}(n - p^{e})
\]
This completes the proof of Lemma~\ref{lem.d(ii)} and thus of
parts (d) of Proposition~\ref{prop1} and of Theorem~\ref{thm1}.
\end{proof}

\end{document}